\documentclass[final, english]{ourlematema}

\usepackage{hyperref}
\usepackage{mathtools} 
\usepackage[shortlabels]{enumitem} 
\usepackage{tikz-cd}
\tikzcdset{scale cd/.style={every label/.append style={scale=#1},
    cells={nodes={scale=#1}}}}
\usepackage{vwcol}  
\usepackage[font=scriptsize]{caption}
\usepackage{comment}
\usepackage{wrapfig}
\usepackage{blkarray}

\newcommand{\Adj}{\mathrm{Adj}}      
\newcommand{\C}{\mathbb{C}}          
\newcommand{\Cat}{\mathrm{Cat}}      
\newcommand{\CQ}{\mathrm{CQ}}        
\newcommand{\Gr}{\mathrm{Gr}}        
\newcommand{\Hilb}{\mathrm{Hilb}}    
\newcommand{\Id}{\mathrm{Id}}        
\newcommand{\ML}{\mathcal{L}}        
\renewcommand{\P}{\mathbb{P}}        
\newcommand{\PGL}{\mathrm{PGL}}      
\DeclareMathOperator{\rk}{rk}        
\newcommand{\sat}{\mathrm{sat}}      
\newcommand{\Sym}{\mathbb{S}}        
\newcommand{\SP}{\mathrm{S}}         

\title{Inverting catalecticants of ternary quartics}

\titlenote{LBM’s contribution has been supported by the Novo Nordisk Foundation grant NNF18OC0052483}


\author{Laura Brustenga i Moncus\'i}
\address{%
Department of Mathematical Sciences\\
University of Copenhagen\\
\email{brust@math.ku.dk}
}

\author{Elisa Cazzador}
\address{%
Department of Mathematics\\
University of Oslo\\
\email{elisacaz@math.uio.no}
}
\authorline
\author{Roser Homs}
\address{%
Department of Mathematics\\
Technical University of Munich\\
\email{roser.homs@tum.de}
}

\begin{document}

\begin{abstract}
We study the reciprocal variety to the LSSM of catalecticant matrices associated with ternary quartics.  
With numerical tools, we obtain 85 to be its degree and 36 to be the ML-degree of the LSSM. 
We provide a geometric explanation to why equality between these two invariants is not reached,
as opposed to the case of binary forms, 
by describing the intersection of the reciprocal variety and the orthogonal of the LSSM in the rank loci.
Moreover, we prove that only the rank-$1$ locus, namely the Veronese surface $\nu_4(\P^2)$, contributes to the degree of the reciprocal variety.
\end{abstract}

\maketitle

\section*{Introduction}
\label{sec:intro}
For a linear subspace $\ML$ of the space $\Sym^m$ of $m\times m$ symmetric matrices over $\C$, the closure of the set $\lbrace A^{-1}\in\Sym^m\mid A\in\ML,\,\det(A)\neq 0\rbrace$ is called the \emph{reciprocal variety} $\ML^{-1}$ to $\ML$.
In this paper, we focus on a concrete linear space of symmetric matrices (LSSM): the space of catalecticant matrices 

\begin{footnotesize}
\begin{equation}
\label{eq:catalecticant-matrix-ternary-quartics}
	\Cat(2,3)=
	\left\lbrace\left(
	\begin{array}{cccccc}
	a_{(4,0,0)} & a_{(3,1,0)} & a_{(3,0,1)} & {\color{blue}a_{(2,2,0)}} & {\color{magenta}a_{(2,1,1)}} & {\color{violet}a_{(2,0,2)}}\\
	a_{(3,1,0)} & {\color{blue}a_{(2,2,0)}} & {\color{magenta}a_{(2,1,1)}} & a_{(1,3,0)} & {\color{teal}a_{(1,2,1)}} & {\color{cyan}a_{(1,1,2)}}\\
	a_{(3,0,1)} & a_{(2,1,1)} & {\color{violet}a_{(2,0,2)}} & {\color{teal}a_{(1,2,1)}} & {\color{cyan}a_{(1,1,2)}} & a_{(1,0,3)}\\
	a_{(2,2,0)} & a_{(1,3,0)} & a_{(1,2,1)} & a_{(0,4,0)} & a_{(0,3,1)} & {\color{purple}a_{(0,2,2)}}\\
	a_{(2,1,1)} & a_{(1,2,1)} & a_{(1,1,2)} & a_{(0,3,1)} & {\color{purple}a_{(0,2,2)}} & a_{(0,1,3)}\\
	a_{(2,0,2)} & a_{(1,1,2)} & a_{(1,0,3)} & a_{(0,2,2)} & a_{(0,1,3)} & a_{(0,0,4)}\\
	\end{array}
	\right)\colon a_{(i,j,k)}\in\C\right\rbrace
\end{equation}
\end{footnotesize}

associated with ternary quartics
\begin{multline*}
a_{(4,0,0)}x^4+a_{(3,1,0)}x^3y+a_{(3,0,1)}x^3z+{\color{blue}a_{(2,2,0)}}x^2y^2+{\color{magenta}a_{(2,1,1)}}x^2yz+{\color{violet}a_{(2,0,2)}}x^2z^2\\
+a_{(1,3,0)}xy^3+{\color{teal}a_{(1,2,1)}}xy^2z+{\color{cyan}a_{(1,1,2)}}xyz^2+a_{(1,0,3)}xz^3+a_{(0,4,0)}y^4\\
+a_{(0,3,1)}y^3z+{\color{purple}a_{(0,2,2)}}y^2z^2+a_{(0,1,3)}yz^3+a_{(0,0,4)}z^4.
\end{multline*}

Catalecticant matrices were introduced by Sylvester in~\cite{sylvester1970principles}. 
They arise in apolarity theory as matrices associated with $k$-th order contractions, see~\cite{IK99}.

The reciprocal variety $\Cat(k,2)^{-1}$ of catalecticants of binary forms of degree $2k$ is the Grassmannian $\Gr(2,k+2)$, see \cite[Proposition 7.2]{ExponentialVarieties}. 
Its degree is the Catalan number $\frac{1}{k+1}\binom{2k}{k}$ and, moreover, it coincides with the ML-degree of $\Cat(k,2)$ (see Definition~\ref{def:ML-degree} and \cite[Proposition 7.4]{ExponentialVarieties}).

Not much is known about reciprocal varieties to catalecticants associated with forms in more than two variables, as the question posed in~\cite[Problem 7.5]{ExponentialVarieties} suggests. 
The goal of our paper is to study the first unknown case: ternary quartics.
The space of catalecticant matrices $\Cat(2,3)\subset\Sym^6$ has dimension $15$. 
The reciprocal variety $\Cat(2,3)^{-1}$ is the closure of the image of $\Cat(2,3)$ via the birational map $\Sym^6\dashrightarrow \Sym^6$ defined by $A\mapsto A^{-1}$. 

Our first main result is Theorem~\ref{thm:numerical-results}:
 using numerical tools implemented in \verb|HomotopyContinuation.jl|~\cite{HomotopyContinuation.jl} and \verb|LinearCovarianceModels.jl|~\cite{LinearCovarianceModels}, 
we see that the reciprocal variety has degree $85$
and the ML-degree of $\Cat(2,3)$ is $36$.
The difference between these last two invariants is a novelty with respect to both catalecticants of binary forms and generic LSSMs, see \cite[Theorem 1]{Sturmfels2010Gaussian}.

A deeper understanding of the reciprocal variety beyond these figures requires a geometric treatment.
The geometric structure of $\P\,\Cat(2,3)$ is well-understood: its rank at most $r$ locus is the $r$-th secant variety of the Veronese surface $\nu_4(\P^2)$, see~\cite{LandsbergOttaviani} and Section~\ref{sec:Graph-Conormal}.
Therefore, Terracini's lemma can be applied to study the images of rank loci of catalecticants via the inversion map, see Proposition~\ref{prop:dimension-fibers-image-strata} and Corollary~\ref{cor:dimension-phi(C_r)}.

The second main result of the paper is Theorem~\ref{prop:degree-formula}. It combines our knowledge on the images of rank loci with intersection theory on complete quadrics to prove that only the rank-$1$ locus --- namely the Veronese surface $\nu_4(\P^2)$ --- contributes to the degree of $\P\,\Cat(2,3)^{-1}$.

The difference between the numerically computed degree and ML-degree in Section~\ref{sec:numerical} is geometrically explained by the non-empty intersection between the reciprocal variety and the orthogonal of $\P\,\Cat(2,3)$ in the rank loci: this is the content of our third main result Proposition~\ref{prop:intersection}.

We conclude with Conjecture~\ref{conj}: the reciprocal variety $\P\,\Cat(2,3)^{-1}$ is cut out by 27 cubic Pfaffians.

The paper is organized as follows.
Section 1 introduces catalecticant matrices and their reciprocal varieties in the projective setting that will be used throughout the paper.
In Section 2 we provide a first approach to understanding the reciprocal variety in which we apply computational techniques to obtain the essential invariants of the variety presented in Theorem~\ref{thm:numerical-results}.
Section 3 is devoted to the study of the image of rank loci via the inversion map.
In Section 4 we bring everything together to provide a better geometric insight on the reciprocal variety and state our main theoretical results: Theorem~\ref{prop:degree-formula} and Proposition~\ref{prop:intersection}.

\subsection*{Acknowledgements} 
We thank Bernd Sturmfels for proposing the interesting problem and Kristian Ranestad for the many helpful comments and insights on its underlying geometry. 
We would like to give thanks to Sascha Timme for his valuable help and assistance regarding numerical computations.
We also thank Tim Seynnaeve, Martin Vodi{\v{c}}ka and Carlos Am{\'e}ndola for both helpful insight and computations. 
Moreover, we thank Marina Garrote for Magma computations. 
Finally, thanks to all organisers and participants of the LSSM project.

\section{Catalecticant matrices and their inverses}
\label{sec:s1}
A square\footnote{Catalecticant matrices are defined for forms of any degree, see~\cite[Definition 1.3]{IK99}. 
However, we are only interested in those which are symmetric, namely associated to forms of even degree.} catalecticant matrix is a symmetric matrix associated to a homogeneous polynomial of even degree as follows.
First, let us get a taste of catalecticants by giving an example.
For a binary quartic of the form
\[
	a_{(4,0)}x_0^4+a_{(3,1)}x_0^3x_1+a_{(2,2)}x_0^2x_1^2+a_{(1,3)}x_0x_1^3+a_{(0,4)}x_1^4,
\]
its associated catalecticant matrix is 

\begin{equation}
\label{eq:catalecticant-matrix-binary-quartics}
\begin{blockarray}{cccc}
	& {(2,0)} & (1,1) & (0,2)\\[2pt]
\begin{block}{c(ccc)}
	(2,0) & a_{(4,0)} & a_{(3,1)} & a_{(2,2)}\\
	(1,1) & a_{(3,1)} & a_{(2,2)} & a_{(1,3)}\\
	(0,2) & a_{(2,2)} & a_{(1,3)} & a_{(0,4)}\\
\end{block}
\end{blockarray}
\end{equation}
where rows and columns are labeled with the bi-degree of quadric monomials
and the entries are indexed by the sum of their row and column indices.

In general, fix two positive integers $k$ and $n$, and set $m=\binom{n+k}{k}$.
This is the number of monomials in $n+1$ variables of degree $k$.
We denote by $\Sym^{m}$ the space of $m\times m$ symmetric matrices.
A symmetric matrix $A\in\Sym^{m}$ has its rows and columns labeled with the monomials of degree $k$.
A matrix $A\in\Sym^{m}$ is a \emph{catalecticant matrix} if its entries are indexed by the product of the monomials represented by the row and column label.
The set $\Cat(k,n+1)$ of all catalecticant matrices is a linear subspace of $\Sym^m$ with dimension $\binom{2k+n}{2k}$ (the number of monomials in $n+1$ variables of degree $2k$).

By considering non-zero matrices in $\Cat(k,n+1)$ and $\Sym^m$ up to scalar multiplication, we obtain the associated projective linear spaces:
\[
  \P\,\Cat(k, n+1) \simeq \P^N \qquad \P\,\Sym^m \simeq \P^M,
\]
where $N=\binom{2k+n}{2k}-1$ and $M=\binom{m+1}{2}-1$. 

\begin{exa}
Binary quartics are associated with catalecticant matrices as in~\eqref{eq:catalecticant-matrix-binary-quartics}. 
Their LSSM is a hyperplane $\Cat(2,2) \subset \Sym^3$, whose single defining equation is given by imposing equality along the main skew-diagonal. 

Ternary quartics are associated with catalecticant matrices as in~\eqref{eq:catalecticant-matrix-ternary-quartics}. 
They define a codimension-$6$ linear space $\Cat(2,3) \subset \Sym^6$.
\end{exa}

We are interested in the \emph{reciprocal variety} of catalecticants, namely
\[
  \Cat(k,n+1)^{-1} := \overline{ \lbrace{ A^{-1}\in\Sym^m \mid A\in\Cat(k,n+1), \det(A) \neq 0 \rbrace} }.
\] and its projectivisation $\P\,\Cat(k,n+1)^{-1}$, which we will also refer to as the reciprocal variety.

For an invertible matrix $A$, its inverse and adjugate are proportional, so they determine the same point in the projective space, that is $[A^{-1}]=[\Adj(A)]\in \P\,\Sym^m$, where we use the square brackets to indicate the class in $\P\,\Sym^m$ of a particular matrix in $\Sym^m$.
Note that $[\Adj(A)]$ is well-defined also for corank-$1$ matrices $A$.

The adjugate map is a natural extension of the inverse map in the projective setting, so we redefine the reciprocal variety in terms of it.
Formally, the space of symmetric matrices $\Sym^m$ is the square symmetric power $\SP^2(V)$ of a $\C$-vector space $V$ of dimension $m$.
In this setting, the adjugate map corresponds to the map $\SP^2(V)\to \SP^2(\wedge^{m-1}V)$ sending $A$ to $\wedge^{m-1}A$.
Note that $\wedge^{m-1}V$ is again a $\C$-vector space of dimension $m$, so the space $\SP^2(\wedge^{m-1}V)$ can also be identified with the space of $m\times m$ symmetric matrices.
We will denote it by $\hat \Sym^m$ in order to distinguish it from $\SP^2(V)$.

\begin{dfn}
\label{defi:inversionMap}
The \emph{inversion map} is the birational morphism $\phi \colon \P\Sym^m \dashrightarrow \P\hat\Sym^m$ defined by $[A] \mapsto [\wedge^{m-1}(A)]$.
\end{dfn}

Birationality follows from $\wedge^{m-1}(\wedge^{m-1}(A)) = \det(A)^{m-1} A$. 
Equivalently, by defining $\hat \phi \colon \P\hat\Sym^m \dashrightarrow \P\Sym^m$ sending $[B]$ to $[\wedge^{m-1} B]$, the restriction of the composite maps $\phi \cdot\hat\phi$ and $\phi \cdot\hat\phi$ over full-rank matrices is the identity map.

\medskip
From now on, we will consider the reciprocal variety $\P\,\Cat(k,n+1)^{-1}$ as the Zariski closure of the image of $\phi$ restricted to $\P\,\Cat(k,n+1)$.
Observe that, since $\phi$ is birational, the dimension of $\P\,\Cat(k,n+1)^{-1}$ is $N$.

\begin{exa}
\label{exa:BinaryQuartics1} 
For binary quartics, the inversion map restricted to $\P\,\Cat(2,2)\simeq \P^4$ is a rational map $\P\,\Cat(2,2) \dashrightarrow \P\,\hat\Sym^3$ with $\P\,\hat\Sym^3\simeq \P^5$.
A point in the image is parametrized by the $2$-minors of the catalecticant matrix
\[
\begin{cases}
	y_0 & = a_{(2,2)}a_{(0,4)} - a_{(1,3)}^2\\
	& \vdots \\
	y_5 & = a_{(4,0)}a_{(2,2)} - a_{(3,1)}^2
\end{cases}
\]
By eliminating the parameters $a_{(i,j)}$ we get the equation for the reciprocal variety $\P\,\Cat^{-1}(2,2) = \{y_2(y_2-y_3) +y_1y_4 -y_0y_5=0\}$. 
This is a quadric hypersurface in $\P\,\Sym^3$ defining $\Gr(2,4)$, the Grassmannian of planes in a $4$-dimensional vector space.

Note that one can alternatively find the quadric equation by computing the preimage of the catalecticant space via the inversion map. 
This amounts to set equality between the $2$-minors 
$\begin{vsmallmatrix}
	y_1 & y_2 \\
	y_3 & y_4
\end{vsmallmatrix}
$
and
$\begin{vsmallmatrix}
	y_0 & y_2 \\
	y_2 & y_5
\end{vsmallmatrix}
$
of the generic $3 \times 3$ symmetric matrix 
$\begin{psmallmatrix}
	y_0 & y_1 & y_2 \\
	y_1 & y_3 & y_4 \\
	y_2 & y_4 & y_5 
\end{psmallmatrix}.
$
\end{exa}

An alternative interpretation of the reciprocal variety of an LSSM $\mathcal{L}$ is inspired by statistics. 
Let $\Sym^m_{>0}$ be the cone of positive definite $m\times m$ matrices. 
$\mathcal{L}^{-1}\cap\Sym^m_{>0}$ is a centered Gaussian statistical model. 
In other words, it encodes a set of multivariate normal distributions $\mathcal{N}(0,\Sigma)$ where the covariance matrix $\Sigma\subset \Sym^m_{>0}$ is defined by linear constrains on the entries of its inverse $K=\Sigma^{-1}$. 
This type of models are known as \emph{linear concentration models}, see~\cite{Sturmfels2010Gaussian}.

\begin{dfn}
\label{def:ML-degree}
The \emph{maximum likelihood degree} (ML-degree) of a linear concentration model is defined as the number of complex solutions to the critical equations of the \emph{log-likelihood function}
\begin{equation}
\label{eq:log}
\ell(K)=\log \det K - \operatorname{trace}\left(SK\right),
\end{equation}
where $S$ is the sample covariance matrix of sample data vectors, see~\cite[Definition 2.1.4]{OWL} for more details and~\cite[
Definition 5.4]{ExponentialVarieties} for an equivalent algebro-geometric definition.
\end{dfn}

Note that~\eqref{eq:log} depends on the random data encoded in the sample covariance matrix $S$. 
Nevertheless this notion of ML-degree is well-defined because the number of complex solutions to the critical equations is preserved for general data vectors, see~\cite[Remark 2.1]{amendola2021maximum}.

The ML-degree of the linear concentration model represented by an LSSM $\ML$ gives a lower bound on the degree of the reciprocal variety $\ML^{-1}$, see ~\cite[Theorem 5.5]{ExponentialVarieties}. 
Equality is reached if and only if 
$\ML^{-1}\cap\ML^{\perp} =  \emptyset$, where the \emph{orthogonal} space $\mathcal{L}^{\perp}$ is defined by 
\[
	\mathcal{L}^{\perp} = \{ Y \in \hat\Sym^6 \mid \text{trace} (A \cdot Y) = 0, \text{ for all } A \in \mathcal{L} \}.
\]
The two invariants coincide in the case of both generic LSSMs and spaces $\Cat(k,2)$ of catalecticant matrices associated to binary forms, see~\cite[Theorem 2.3]{Sturmfels2010Gaussian} and~\cite[Proposition 7.4]{ExponentialVarieties}, respectively.

\begin{exa}
\label{exa:BinaryQuarticsML} 
To compute the ML-degree of the model $\Cat(2,2)\cap\Sym^3_{>0}$ we replace $K$ and $S$ in 
the log-likelihood function~\eqref{eq:log} by the catalecticant matrix of binary quartics~\eqref{eq:catalecticant-matrix-binary-quartics} and the sample covariance matrix $S=\frac{1}{3}XX^t$, where $X$ is a $3\times 3$ matrix whose columns are random data vectors.   
Its critical equations can be computed with the code below in the algebraic software \verb|Macaulay2|~\cite{M2}

\begin{scriptsize}
	\begin{verbatim}
	I=ideal{jacobian(matrix{{det K}})-(det K)*jacobian(matrix{{trace(S*K)}})}
	J=saturate(I,det K)
	\end{verbatim}
\end{scriptsize}

\noindent
and indeed the degree of the zero-dimensional ideal $J$ coincides with the degree of $\P\,\Cat(2,2)^{-1}=\Gr(2,4)$.
\end{exa}

For $n\ge 2$, the variety $\Cat(k,n+1)^{-1}$ has not been identified with any classical construction, and even their degree and ML-degree are unknown in general.
The study case in the current paper is the space of catalecticants of ternary quartics, $\P\,\Cat(2,3)\simeq \P^{14}$. 
Its inversion map is $\P\,\Cat(2,3) \dashrightarrow \P\,\hat\Sym^6$ with $\P\,\hat\Sym^6\simeq \P^{20}$ and the reciprocal variety $\P\,\Cat(2,3)^{-1}\subseteq \P\,\hat\Sym^6$ is 14-dimensional.

\section{First steps}
\label{sec:numerical}
In this section, we sketch a natural approach to compute the equations of the reciprocal variety $\P\,\Cat(k,n+1)^{-1}$, which turns out to be computationally unfeasible.
We also present the computations of some invariants by means of numerical methods which will give some insight in the forthcoming sections: the degree of $\P\,\Cat(2,3)^{-1}$, the number of linearly independent equations for the reciprocal variety in low degrees and the ML-degree of the linear concentration model represented by $\Cat(2,3)$.
The results are summarized in Theorem~\ref{thm:numerical-results}.
\medskip

Proposition~\ref{prop:Elisas-ideal-general} below shows a natural approach to obtain the equations of $\P\Cat(k,n+1)^{-1}$ based on the fact that $\hat\phi \circ \phi = \Id_{\P\Sym^m}$.

\begin{prop}
\label{prop:Elisas-ideal-general}
Consider the maps $\phi\colon \P\,\Sym^m\dashrightarrow \P\,\hat\Sym^m$ and $\hat\phi\colon\P\,\hat\Sym^m\dashrightarrow \P\,\Sym^m$, and let $U\subseteq \P\,\Sym^m$ be the open subset of invertible symmetric matrices.
For every subvariety $Y$ of $\P\,\Sym^m$, set-theoretically we have
\[
	\hat\phi^{*}(Y) = \phi(Y\cap U)\cup \bigl(\hat\phi^{*}(Y)\cap V(\textstyle \det_{\Sym})\bigr)
\] 
where $\hat\phi^{*}(Y)$ denotes the pullback of $Y$ by $\hat\phi$ and $\det_{\Sym}$ is the determinant polynomial for symmetric matrices of $\P\,\hat\Sym^m$.

In particular, given ideals $I$ and $J$ defining $\overline{\phi(Y\cap U)}$ and $\hat\phi^{*}(Y)$
\[
	I = \sqrt{\sat(J, \textstyle\det_{\Sym})}.
\]
\end{prop}

\begin{proof}
  Consider $\hat U\subseteq \P\,\hat\Sym^m$ the open subset of invertible symmetric matrices.
  Set-theoretically, we have
\[
   \hat\phi^{*}(Y) =  \bigl(\hat\phi^{*}(Y)\cap \hat U\bigr)\cup \bigl( \hat\phi^{*}(Y)\cap V(\textstyle \det_{\Sym})\bigr).
 \] Now, $\hat\phi^{*}(Y)\cap \hat U=\phi(Y\cap U)$ follows straightforwardly from the fact that $\phi|_U$ and $\hat\phi|_{\hat U}$ are mutually inverse.
\end{proof}

In our particular case of catalecticant matrices $\P\,\Cat(k,n+1)$ in the space of symmetric matrices $\P\,\Sym^m$, Proposition~\ref{prop:Elisas-ideal-general} gives, first, equations that $\P\,\Cat(k,n+1)^{-1}$ must satisfy and, second, a procedure to compute its exact equations.
Namely, the set of symmetric matrices $S$ with catalecticant inverse, that is $\phi(S)\in\P\,\Cat(k,n+1)$, contains $\P\,\Cat(k,n+1)^{-1}$.
Such a set is defined by the linear relations on the maximal minors of $S$ given by the condition $\phi(S)\in \P\,\Cat(k,n+1)$.
These relations define the ideal $J$, and
\begin{equation}
\label{eq:cat=Elisas-quot-det}
 	\P\,\Cat(k,n+1)^{-1}=\overline{V(J)\setminus V(\textstyle\det_{\Sym})}.
\end{equation}
Computing the saturation of $J$ with respect to $\det_{\Sym}$, we could obtain the equations for $\P\,\Cat(k,n+1)^{-1}$, but this turns to be computationally unfeasible for ternary quartics. 

\paragraph*{Degree of the reciprocal variety} 
Equation~\eqref{eq:cat=Elisas-quot-det} allows us to use the ideal $J$ to compute the degree of $\P\,\Cat(2,3)^{-1}$ by means of numerical methods.
Indeed, for a general linear space $L\subseteq\P\,\hat\Sym^6$ of codimension $\dim(V(J))=14$, we have $\dim(V(J)\cap L)=0$ and
\[
	(V(J)\cap L)\cap V(\textstyle\det_{\Sym}) = \emptyset.
\] 
Hence, $\deg(\P\,\Cat(k,n+1)^{-1})$ is equal to the number of points in $V(J)\cap L$.
We use the monodromy method, implemented in \verb|HomotopyContinuation.jl|~\cite{HomotopyContinuation.jl}, setting as parameters the coefficients of $L$ to compute the degree of $\P\,\Cat(2,3)^{-1}$.
We obtain
\[
 	\deg(\P\,\Cat(2,3)^{-1})=85.
\]

\paragraph*{Number of equations in low degree} 
An invariant we may numerically estimate is, for a fixed degree $d$, the number $HF_d$ of linearly independent forms in $I$.
For example, for $d=3$, the space $\Omega$ of forms in $\dim(\hat\Sym^6)=21$ variables of degree $3$ has dimension $\binom{3+21-1}{3}=1771$.
Given a symmetric matrix $S$, the set of forms $F\in\Omega$ for which $F(S)=0$ gives rise to a hyperplane $H_S\subseteq \Omega$ defined by the linear form with coefficients these $1771$ monomials evaluated at $S$.
Thus, for a generic set of symmetric matrices $\{S_1,\dots,S_{1771}\}$, we have
\[
 	\bigcap_{i=1}^{1771}H_{S_i}=\{F_0\}.
\] 
Instead, for a generic set of points $\{P_1,\dots,P_{1771}\}$ of $\P\,\Cat(2,3)^{-1}$, if there are forms of degree $3$ in $I$, the hyperplanes $H_{P_i}$ will exhibit some linear dependence, coming from the fact that there would be a linear space of positive dimension lying in the intersection of all of them, hence
\[
	HF_3=\dim(\bigcap_{i=1}^{1771}H_{P_i}).
\]
Therefore, we may compute $HF_3$ as the rank of the $1771\times 1771$ matrix of coefficients of the linear forms defining $H_{P_i}$.

Using \verb|Julia|, we sampled points in $\P\,\Cat(2,3)^{-1}$ as inverses of randomly generated catalecticant matrices, and we obtained $HF_d=0,27,510$ for $d=2,3,4$.

\paragraph*{Maximum likelihood degree} 
The computational cost for the required symbolic computation of the critical points of~\eqref{eq:log} analogous to Example~\ref{exa:BinaryQuarticsML} is too high already for ternary quartics in \verb|Macaulay2|.
The software \verb|Magma|~\cite{Magma}, allows us to perform a symbolic computation\footnote{These computations took an average of 5 days and were performed with Magma V2.25-5 on a Dual Core Intel(R) Xeon(R) (2.20 GHz).} in which we obtain that the ML-degree of the model represented by $\Cat(2,3)$ is 36.

Note that, for a sample covariance matrix $S$ associated to general data vectors, all complex solutions to~\eqref{eq:log} are different, i.e. they have multiplicity one, see~\cite[Remark 2.1]{amendola2021maximum}. Therefore, a numerical approach is possible and it is already offered by the \verb|Julia| package \verb|LinearCovarianceModels.jl|~\cite{LinearCovarianceModels}. Note that our desired value corresponds to the dual ML-degree of $\ML$ defined in~\cite{Sascha}. 

The code used in \verb|Macaulay2| and \verb|Julia| to compute ML-degrees of LSSMs of catalecticant matrices can be found in~\cite{Complementary}.

We summarize all the numerical results as follows:

\begin{thm}
\label{thm:numerical-results}
The degree of the reciprocal variety $\P\,\Cat(2,3)^{-1}$ is $85$, whereas the ML-degree of the linear concentration model $\P\,\Cat(2,3)$ is $36$.
Moreover, there are $27$ cubics among the minimal generators of the defining ideal of the reciprocal variety.
\end{thm}

\begin{rem}
The numerical approach for ternary forms of higher degree is not feasible for short-time computations. 
At the level of ternary quartics, the degree (and ML-degree) of a generic $15$-dimensional LSSM of $6\times 6$ matrices is $1016$, see~\cite{manivel2020complete} and the implementation of the algorithm to compute the degree of the reciprocal variety of a generic LSSM in~\cite{MPIrepo}.
For ternary sextics, the degree of a generic $28$-dimensional LSSM of $10\times 10$ matrices is $17.429.229.428$. 
The ML-degree of $\Cat(3,3)$ is already lower bounded by $180.000$.
\end{rem}

\section{Inverting the rank loci}
\label{sec:Graph-Conormal}
Let us denote by $C_r$ the locus of matrices of rank at most $r$ in $\P\,\Cat(2,3)$. 
The rank loci give a stratification $C_1 \subseteq \ldots \subseteq C_6=\P\,\Cat(2,3)$ and have a well-known geometric structure:
$C_r$ is the $r$-secant variety of the Veronese surface $v_{4}(\P^2)$. 

The goal of this section is to study the image of each $C_r$ via the inversion map $\phi$ in the sense of~\eqref{eq:imageRankLoci}.
The main result is Proposition~\ref{prop:dimension-fibers-image-strata}, which provides a description of the image of any rank-$r$ point in $\P\,\Cat(2,3)$.

On one hand, this allows us to compute all dimensions of the image of rank loci, see Corollary~\ref{cor:dimension-phi(C_r)}.
This will be crucial later in Proposition~\ref{prop:degree-formula} to show how the rank loci contribute to the degree of the reciprocal variety $\P\,\Cat(2,3)^{-1}$.

On the other hand, it provides explicit equations for the image of rank-$2$ points. 
This will help to understand the intersection of the reciprocal variety with the orthogonal space of $\P\,\Cat(2,3)$, see Proposition~\ref{prop:intersection}.

\paragraph*{Geometry of the rank loci}
Let us start with a focus on the rank loci of $\P\,\Cat(2,3)$. In~\cite{LandsbergOttaviani} it is proven that for every $r$ the locus  $C_r$ corresponds to the $r$-secant variety  $\sigma_r(v_{4}(\P^2))$ of the Veronese surface. 
Their dimensions are known by Alexander-Hirshowitz theorem~\cite{brambilla2008alexander} and they are equal to $2, 5, 8, 11, 13$ for $r=1,2,3,4,5$, respectively.
These dimensions already indicate that the catalecticant space is far form being general: indeed, for a generic $15$-dimensional LSSM $\mathcal{L}$, the rank loci in $\P\mathcal{L}$ would have dimension
$-1$ (the emptyset), $4,8,11$ and $13$, respectively.

The relation between rank loci and secant varieties of Veronese embeddings rarely occurs, see~\cite{BucBuc},~\cite{LandsbergOttaviani}.
This is a partial obstacle to the generalisation of our results to any $\P\,\Cat(k,n+1)$.

\medskip
The notion of image via a rational map can be extended for a point in its base locus, see~\cite[p. 75]{harris1992first}. 
In our setting, if $\Gamma$ is the closure of the graph of $\phi$ in $\P\,\Cat(2,3) \times \P\hat\Sym^6$ and $\pi_1, \pi_2$ are the projection maps from this product, we may write
\begin{equation}
\label{eq:imageRankLoci}
	\phi(X) := \pi_2( \pi_1^{-1}(X) \cap \Gamma )\subseteq \P\hat\Sym^6.
\end{equation}

\begin{rem}
\label{rem:relation-blowup-graph}
For every subvariety $X\subset \P\,\Cat(2,3)$, its image $\phi(X)$ can be understood in the following way: we may regularize $\phi$ by a (finite) sequence of blow-ups along smooth centers, and consider the strict transform of $X$ by such a sequence of blow-ups.
Its image via the regularized map will coincide with $\phi(X)$.
In Section~\ref{sec:variety}, we will see that a natural way of resolving the indeterminacy locus of $\phi$ is by blowing-up the locus of rank-$1$ symmetric matrices, then blowing-up the strict transform of rank-$2$ symmetric matrices and so on.
The relation between blowing-up rank loci and the graph $\Gamma$ will make the varieties $\phi(C_r)$ play a central role in computing the degree of the whole $\P \,\Cat(2,3)^{-1}$.
\end{rem}

We give here two remarks: the first on how $\phi$ affects the rank,  the second addressing a particular property of the rank loci $C_r$ as a consequence of being secant varieties.

\begin{rem}
\label{rem:rank-to-corank}
The inversion map $\phi$ sends rank-$r$ matrices to symmetric matrices of rank at most $6-r$.
This follows by noting that every pair of points $(A,B) \in \Gamma$ satisfies $AB = 0$.
\end{rem}

\begin{rem}
\label{rem:PGL(3)-orbits}

The natural action of $\PGL(3)$ on $\P^2$ induces an action on the Hilbert scheme of points $\Hilb_r(\P^2)$, which has an open orbit when $r \leq 4$, namely the orbit of $r$ general points.
\end{rem}

\paragraph*{Image of the rank loci}
We next provide a procedure to obtain parametrizations for the image of a rank-$r$ point in any LSSM, and use this to study the image of catalecticant rank loci.

\begin{lemma}
\label{lem:parametrization-with-limits}
Let $L \subset \P\Sym^m$ be any projective LSSM with rank loci $L_1, \ldots L_m$ and let $A \in L_r \setminus L_{r-1}$ be any rank-$r$ point.
Then the image of $A$ via the inversion map is, set-theoretically,
\begin{equation}
\label{eq:limit-formula}
	\overline{ 
	\{ \lim\nolimits_{t\to 0}\phi(A + t X) \mid X \in L, \, \det X \neq 0 \} 
	}.
\end{equation}
Moreover, the above set of limits only depends on the normal space $N_{L_r}$ to $L_r$ at $A$.
In particular, $\dim \phi(A) \leq \dim N_{L_r|_A} - 1$.
\end{lemma}

\begin{proof} 
By continuity, the points in $\phi(A)$ are limits of images of points approaching  $A$ along directions where $\phi$ is well-defined, that is~\eqref{eq:limit-formula}. 

Recall that, by definition, $\phi(A)$ is $\pi_2(\pi_1^{-1}(A) \cap \Gamma)$. Remark~\ref{rem:relation-blowup-graph} implies that $\pi_1^{-1}(A) \cap \Gamma =  \pi_1^{-1}(A) \cap L_r^{(r)}$, 
where $L_r^{(s)}$ is the proper transform of $L_r$ in the blow-up of $\P\Sym^m$ along the symmetric loci of ranks $1, \ldots, s$. 
Equivalently, $L_r^{(r)}$ is the projectivization $\P N_{L_r^{(r-1)}}$ of the normal bundle of $L_r^{(r-1)}$ in the $(r-1)$-th blow-up of $\P\Sym^m$. 
In particular, the dimension of $\pi_1^{-1}(A) \cap L_r^{(r)}$ is the fiber dimension of $\P N_{L_r^{(r-1)}}$ at $\pi_1^{-1}(A)$, which might drop after projecting via $\pi_2$. 
\end{proof}

\begin{prop}
\label{prop:dimension-fibers-image-strata}
The image of a rank-$r$ point in $\P\,\Cat(2,3)$ via $\phi$ is
\begin{itemize}[label={--}]
	\itemsep-0.3em
	\item a $\P^5$, a $\P^2$ and a point for $r=3,4,5$, respectively;
	\item a cubic hypersurface embedded in a $\P^9$ for $r=2$;
	\item an $11$-fold embedded in a $\P^{14}$ for $r=1$.
\end{itemize}
Moreover, for each rank-$2$ point, the associated cubic hypersurface is defined by the cubic Pfaffian of a $6\times 6$ skew-symmetric matrix.
\end{prop}

\begin{proof}
The case $r=5$ is trivial. 
By Lemma~\ref{lem:parametrization-with-limits}, with $L = \P\,\Cat(2,3)$, the dimension of $\phi(A)$ is bounded by $\dim \P\,\Cat(2,3)-\dim T_A C_r-1$, that is $\dim \phi(A) \leq 2$ (resp. $ 5, 8, 11$) when $\rk (A)=4$ (resp. $3,2,1$). 
We use~\eqref{eq:limit-formula} to build a parametrization for $\phi(A)$ and use \verb|Macaulay2| to find its equations via elimination.

The result follows for $r=3,4$ as soon as it is checked on one general point of rank $r$.
Indeed, $\phi(A)$ only depends on the normal space $N_{C_r|_A}$ (hence on its tangent space) so by Terracini's Lemma~\cite{terracini-1911-sulle-vk}, $\phi$ is constant on each $r$-secant space. 
Moreover, Remark~\ref{rem:PGL(3)-orbits} implies that the image of any two points on a general $r$-secant are projectively equivalent.
When $A$ is a rank-$r$ point on a degenerate secant, the dimension of $\phi(A)$ can only increase, which the above bound does not allow. 
Similarly, the degree of $\phi(A)$ can only decrease at special points, when the dimension does not increase.  

We explain in more detail the case $r=2$, and show that the associated cubic hypersurface is a Pfaffian both for the open orbit of proper secant lines and the closed orbit of tangent lines.

In the first case, we may assume that the point $A$ lies on the secant line $\langle P,Q \rangle$, with $P=v_4([1:0:0])$ and $Q=v_4([0:0:1])$. 
Without loss of generality we can take the catalecticant matrix whose entries are all set equal to zero, except for $a_{(4,0,0)} = a_{(0,0,4)} = 1$. 
If we set $y_0, \ldots y_{20}$ to be the projective coordinates for 
matrices 
$
\begin{psmallmatrix}
	y_0 & y_1 & y_2 & y_3 & y_4 & y_5 \\
	\cdots & y_6 & y_7 & y_8 & y_9 & y_{10} \\
	\cdots & \cdots & \cdots & \cdots & \cdots & \cdots \\
	\cdots & \cdots & \cdots & \cdots & \cdots & y_{20} \\
\end{psmallmatrix}
$
in $\P\hat\Sym^6$, elimination gives $11$ linear equations 
\[
  {y}_{20}={y}_{19}={y}_{17}={y}_{14}={y}_{10}={y}_{5}={y}_{4}={y}_{3}={y}_{2}={y}_{1}={y}_{0}=0
\]
together with the cubic Pfaffian of the matrix $S_1$ below.

In the second case, we may assume that $A$ is a point on a tangent line to $P$, and a representative for it is the matrix $A$ for which $a_{(4,0,0)}=a_{(3,1,0)}=1$ and the remaining entries are zero. 
We get again $11$ linear equations
\[
    {y}_{10}={y}_{9}={y}_{8}={y}_{7}={y}_{6}={y}_{5}={y}_{4}={y}_{3}={y}_{2}={y}_{1}={y}_{0}=0
\]
together with the cubic Pfaffian of the matrix $S_2$ below.
\[
S_1 = 
\begin{psmallmatrix}
	0 & 0 & y_6 & y_7 & y_8 & y_9 \\
	0 & 0 & y_7 & y_{11} & y_{12} & y_{13} \\
	-y_6 & -y_7 & 0 & y_{12}-y_9 & y_{15} & y_{16} \\
	-y_7  & -y_{11} & y_9-y_{12} & 0 & y_{16} & y_{18} \\
	-y_8 & -y_{12} & -y_{15} & -y_{16} & 0 & 0 \\
	-y_9 & -y_{13} & -y_{16} & -y_{18} & 0 & 0
\end{psmallmatrix}
\quad
S_2 = 
\begin{psmallmatrix}
	0 & 0 & y_{19} & y_{20} & y_{14} & y_{17} \\
	0 & 0 & y_{13} & y_{14} & y_{11} & y_{12} \\
	-y_{19} & -y_{13} & 0 & y_{17}-y_{18} & y_{12} & y_{15} \\
	-y_{20}  & -y_{14} & y_{18}-y_{17} & 0 & y_{13} & y_{16} \\
	-y_{14} & -y_{11} & -y_{12} & -y_{13} & 0 & 0 \\
	-y_{17} & -y_{12} & -y_{15} & -y_{16} & 0 & 0
\end{psmallmatrix}
\]

When $r=1$, the $6$ linear equations defining the $\P^{14}$ in which $\phi(A)$ is embedded can be easily read from the parametrization of its image.  
The statement regarding dimension is proven numerically, again with \verb|Macaulay2|, by intersecting the parametrization of $\phi(A)$ with $11$ general hyperplanes of $\P\hat\Sym^m$, which gives $14$ points.

The code for all these computations can be found in~\cite{Complementary}. 
\end{proof}

\begin{cor}
\label{cor:dimension-phi(C_r)}
For $r=1,\dots,4$, we have $\dim \phi(C_r) = 2r + 13 - \dim C_r$, while $\dim(C_5)=5$. 
In particular, only $\phi(C_1)$ is a hypersurface in the reciprocal variety.
\end{cor}

\begin{proof}
By Proposition~\ref{prop:dimension-fibers-image-strata}, the image of a point $A \in C_r$ is a variety of dimension $13-\dim C_r$. 
Terracini's Lemma ensures that for any two points $A_1, A_2$ in the same $r$-secant space $S_r$, we have $\phi(A_1) = \phi(A_2)$. 
Then 
\[
	\phi(C_r) = \bigcup_{A \in C_r} \phi(A) = \bigcup_{S_r \subset \sigma_r(v_4(\P^2))} \phi(S_r).
\]
An $r$-secant space is determined by the choice of $r$ points on the  surface, so altogether $\dim \phi(C_r) = 2r + 13 - \dim C_r$.

The case $r=5$ is special because $5$ points on $v_4 (\P^2)$ uniquely determine the tangent space to $\sigma_5(v_4(\P^2))$, again by Terracini's Lemma. 
This tangent space is in fact a hyperplane in $\P\,\Cat(2,3)$, which cuts the Veronese surface in a curve that is the image, via the Veronese embedding, of a degree $4$ curve in $\P^2$. 
In particular, such curve is singular at those $5$ points, so it must be the image of a double conic. 
Therefore, for $r=5$ the choice of a tangent space is uniquely determined by the choice of a point in the $\P^5$ of plane quadrics. 
\end{proof}

\section{Towards understanding the reciprocal variety}
\label{sec:variety}
In this last section, we bring together the results from Section~\ref{sec:Graph-Conormal} to move towards a better understanding of the reciprocal variety $\P\,\Cat(2,3)^{-1}$.

The dimension analysis made in Proposition~\ref{prop:dimension-fibers-image-strata}, combined with basic intersection theory of complete quadrics, allows us to prove that the degree of the reciprocal variety depends only on the degree of $\phi(C_1)$, see Theorem~\ref{prop:degree-formula}.

The numerical results obtained in Section~\ref{sec:numerical}, imply that there must be non-empty intersection between $\P\,\Cat(2,3)^{-1}$ and the orthogonal space $\P\,\Cat(2,3)^\perp$. 
In Proposition~\ref{prop:intersection} we explain this phenomenon from a geometric perpective, by studying the intersection  of each $\phi(C_r)$ with the orthogonal space.

\paragraph*{Contribution of the rank loci to the degree}
Complete quadrics provide a natural set up when it comes to understanding the inversion map and its regularization.
We give a quick overview of the needed definitions and properties  in the case of catalecticants of ternary quartics.
For a more general and detailed treatment of this topic, see~\cite[Section 3]{manivel2020complete} as well as the references cited therein.

Set $V = \C^6$  and recall from Section~\ref{sec:s1} that $\P(\SP^2(V))$ is the space $\P \Sym^6$ of symmetric matrices while $\P(\SP^2(\wedge^{5}V))=\P\hat\Sym^6$. 

The space of \emph{complete quadrics} on $V$, denoted $\CQ(V)$ is the Zariski closure of the image of
\[
 	\psi \colon \P(\SP^2(V))) \dashrightarrow \P(\SP^2(V)) \times \P \Bigl(\SP^2\Bigl(\bigwedge^2 V \Bigr) \Bigr) \times \cdots \times \P \Bigl(\SP^2 \Bigl(\bigwedge^{5}V \Bigr) \Bigr)
\]
given by 
\[
 	[A] \mapsto ([A], [\wedge^2 A], \ldots, [\wedge^{5} A]).
\]

We consider two kinds of divisor classes in the intersection ring of $\CQ(V)$.
Classes of the first kind, denoted by $\mu_1, \ldots, \mu_{5}$, are the pull-backs of hyperplane classes via the natural projection maps $\pi_1, \ldots, \pi_5$ on the five factors.

Classes of the second kind, denoted by $\delta_1, \ldots, \delta_{5}$, correspond to exceptional divisors coming from the blow-up of the rank loci.
We are going to use the following relation between divisor classes:

\begin{equation}
\label{eq:divisor-class-relation}
	\mu_{5} = \frac{1}{6}\cdot \sum_{r=1}^{5} r \cdot \delta_r.
\end{equation}

\begin{rem}
Complete quadrics can be equivalently thought as a particular compactification of the space of smooth quadrics in $\P V$. 
This is done by blowing-up $\P\Sym^6$ first along its rank-$1$ locus, then the strict transform of its rank-$2$ locus, and so on up to the rank-$5$ locus.
In particular, $\pi_{5}$ is a regularization of $\phi \colon \P\Sym^6 \dashrightarrow \P\hat\Sym^6$.
\end{rem}

Let us denote with $\widetilde{\P\,\Cat}$ the proper transform of $\P\,\Cat(2,3)$  inside $\CQ(V)$, and with $\widetilde{C}_r$ the strict transform of $C_r$.
The projection map $\pi_{5}$ restricted to $\widetilde{\P\,\Cat}$ is then a regularization of $\phi \colon \P\,\Cat(2,3) \dashrightarrow \P\hat\Sym^6$.

\begin{thm}
\label{prop:degree-formula} 
The degree of the reciprocal variety of $\P\,\Cat(2,3)$ only depends on the degree of $\phi(C_1)$.
More precisely,
\[
  \deg ( \P\,\Cat(2,3)^{-1}) = \frac{\deg \phi(C_1)}{6}.
\]
\end{thm}

\begin{proof}
The degree of $\P\,\Cat(2,3)^{-1}$ is, by definition, the number of points lying in its intersection with $14$ general hyperplanes of $\P\hat\Sym^6$.
Since $\phi$ is birational, its regularization has degree $1$.
Hence, by pulling back classes in the space of complete quadrics, the degree of the reciprocal variety can be equivalently computed as 
$[\widetilde{\P\,\Cat}]\cdot \mu_{5}^{14}$. 
According to~\eqref{eq:divisor-class-relation} 
we can rewrite the last expression as 
\[
  \frac{\sum_{r=1}^{5} r \cdot[\widetilde{\P\,\Cat}]\cdot\delta_r\cdot\mu_{5}^{13}}{6} = 
  \frac{\sum_{r=1}^{5}r \cdot[\widetilde{C_r}]\cdot\mu_{5}^{13}}{6}.
\]
But $[\widetilde{C_r}]\cdot\mu_{5}^{13}$ coincides with the number of points in the intersection of $\phi(C_r)$ with $13$ general hyperplanes of $\P\hat\Sym^6$. 
Corollary~\ref{cor:dimension-phi(C_r)} implies that this number of points is $0$ for every $r>1$.
So, the unique contribution to the above sum is $[\widetilde{C_1}]\cdot \mu_{5}^{13} = \deg \phi(C_1)$.
\end{proof}

For ternary quartics and binary forms of even degree only $\phi(C_1)$ contributes to the degree of the reciprocal variety.
This might not be true for higher catalecticants, where the rank loci are not necessarily secant varieties to Veronese embeddings.

\begin{que}
Does the degree of $\P\,\Cat(k,n+1)^{-1}$ depend only on the degree of $\phi(C_1)$, for every $(k,n)$?
\end{que}

\paragraph*{Intersection with the orthogonal space}
The orthogonal space $\P\,\Cat(2,3)^\perp$ is a $5$-dimensional projective linear space. 
We want to study its intersection with $\P\,\Cat(2,3)^{-1}$ and the first step towards this is the analysis of its rank loci.

\begin{rem}
\label{rem:rank-loci-of-orthogonal}
A computation with \verb|Macaulay2| shows that the generic rank of $\P\,\Cat(2,3)^{\perp}$ is $6$ and that it is empty in rank $1$ and $2$. 
On the other hand, the rank-$3$ locus is a Veronese surface $v_2(\P^2)$ while the loci of rank $4$ and $5$ are the cubic hypersurface defining the secant variety $\sigma_2(v_2(\P^2))$.
The $\PGL(3)$ action in Remark~\ref{rem:PGL(3)-orbits} ensures that the only possibilities for $\phi(C_r)\cap \P\,\Cat(2,3)^{\perp}$ are the rank loci of the orthogonal space.
\end{rem}

\begin{prop}
\label{prop:intersection}
The orthogonal $\P\,\Cat(2,3)^{\perp}$ does not contain the image of any full-rank point and intersects the varieties  $\phi(C_r)$ in the emptyset for $r=3,4,5$, and set-theoretically in a Veronese surface $v_2(\P^2)$ for $r=1,2$. 
In particular: 
\[
	\P\,\Cat(2,3)^{-1}\cap \P\,\Cat(2,3)^{\perp} \neq \emptyset.
\]
\end{prop} 

\begin{proof}
We give an account of the computations that are needed for each case. 
A \verb|Macaulay2| code for this is available at~\cite{Complementary}. 

We first compute the pull-back $\hat\phi^{*}(\P\,\Cat(2,3))$ and verify that its intersection with  $\P\,\Cat(2,3)^{\perp}$ happens only along degenerate matrices.

When $r=4,5$, the statement follows from Remark~\ref{rem:rank-loci-of-orthogonal}: indeed, points on $\phi(C_4)$ and $\phi(C_5)$ have rank at most $2$ and $1$ respectively, while $\P\,\Cat(2,3)^{\perp}$ is empty in rank $1$ and $2$. 

The cases $r=2,3$ can be checked fiberwise using Proposition~\ref{prop:dimension-fibers-image-strata}. 
For any rank-$3$ point, and for rank-$2$ points lying on proper secant lines to $C_1$, the corresponding $\P^5$s and cubic $8$-folds intersect $\P\,\Cat(2,3)^{\perp}$ in the emptyset. 
On the other hand, for rank-$2$ points lying on tangent lines to $C_1$, the corresponding $8$-fold intersects $\P\,\Cat(2,3)^{\perp}$ in one single point belonging to its Veronese surface $v_2(\P^2)$. 
This intersection point varies inside the surface as we consider different tangent lines. 
We then concude by Remark~\ref{rem:rank-loci-of-orthogonal}.

Finally, for $r=1$, we check that the image of $v_4([1:0:0])$ contains the point $[0: \cdots : 0 : 1 :-2:0:0]$.
We again conclude by Remark~\ref{rem:rank-loci-of-orthogonal}. 
\end{proof}

We conclude with some final conjectures and observations.
From Proposition~\ref{prop:dimension-fibers-image-strata} we know that for a point $A \in C_2$, its image $\phi(A)$ is defined by the cubic Pfaffian of a $6 \times 6 $ skew-symmetric matrices. 
A dimension count would suggest that something similar happens also for rank-$1$ points:

\begin{conj}
If $A$ is a point on $C_1$, then $\phi(A)$ is defined by the $7$ cubic Pfaffians of a $7 \times 7 $ skew-symmetric matrix.
\end{conj}

We now might wonder if the reciprocal variety itself is defined by the cubic Pfaffians of a $8 \times 8$ skew-symmetric matrix. 
Although the codimension would be the correct one, the degree would be at most $84$ and the number of cubic generators would be $28$, both numbers differing from the numerical results obtained in Section~\ref{sec:numerical}. 
What we can still conjecture is:

\begin{conj}
\label{conj}
The reciprocal variety $\P\,\Cat(2,3)^{-1}$ is defined by exactly $27$ cubic equations which are Pfaffians of (possibly different) $7 \times 7$ skew-symmetric matrices.
\end{conj}


\bibliography{references}

\providecommand{\bysame}{\leavevmode\hbox to3em{\hrulefill}\thinspace}
\providecommand{\MR}{\relax\ifhmode\unskip\space\fi MR }
\providecommand{\MRhref}[2]{%
  \href{http://www.ams.org/mathscinet-getitem?mr=#1}{#2}
}
\providecommand{\href}[2]{#2}
\begin{thebibliography}{10}

\bibitem{amendola2021maximum}
Carlos Am\'endola - Lukas Gustafsson - Kathl\'en Kohn - Orlando Marigliano -
  Anna Seigal, \emph{The Maximum Likelihood Degree of Linear Spaces of
  Symmetric Matrices}, arXiv:2012.00198 (2021).

\bibitem{Magma}
Wieb Bosma - John Cannon - Catherine Playoust, \emph{The {M}agma algebra
  system. {I}. {T}he user language}, J. Symbolic Comput. {24} no.~3-4 (1997),
  235--265, Computational algebra and number theory (London, 1993).

\bibitem{brambilla2008alexander}
Maria~Chiara Brambilla - Giorgio Ottaviani, \emph{On the Alexander--Hirschowitz
  theorem}, Journal of Pure and Applied Algebra {212} no.~5 (2008), 1229--1251.

\bibitem{HomotopyContinuation.jl}
Paul Breiding - Sascha Timme, \emph{{H}omotopy{C}ontinuation.jl: {A} {P}ackage
  for {H}omotopy {C}ontinuation in {J}ulia}, International Congress on
  Mathematical Software, Springer, 2018, pp.~458--465.

\bibitem{Complementary}
Laura Brustenga~i Moncus\'i - Elisa Cazzador - Roser Homs, \emph{Complementary
  material to "Inverting catalecticants of ternary quartics"},  (2020),
  \url{https://github.com/LauraBMo/InvertingCats}.

\bibitem{BucBuc}
Weronika Buczy\'{n}ska - Jaros{\l}aw Buczy\'{n}ski, \emph{Secant varieties to
  high degree {V}eronese reembeddings, catalecticant matrices and smoothable
  {G}orenstein schemes}, J. Algebraic Geom. {23} no.~1 (2014), 63--90.

\bibitem{OWL}
Mathias Drton - Bernd Sturmfels - Seth Sullivant, \emph{Lectures on algebraic
  statistics}, Oberwolfach Seminars, vol.~39, Birkh\"{a}user Verlag, Basel,
  2009.

\bibitem{M2}
Daniel~R. Grayson - Michael~E. Stillman, \emph{Macaulay2, a software system for
  research in algebraic geometry}, Available at
  \url{http://www.math.uiuc.edu/Macaulay2/}.

\bibitem{harris1992first}
Joe Harris, \emph{Algebraic geometry: a first course}, vol. 133, Springer
  Science \& Business Media, 2013.

\bibitem{IK99}
Anthony Iarrobino - Vassil Kanev, \emph{Power sums, {G}orenstein algebras, and
  determinantal loci}, Lecture Notes in Mathematics, vol. 1721,
  Springer-Verlag, Berlin, 1999, Appendix C by Iarrobino and Steven L. Kleiman.

\bibitem{LandsbergOttaviani}
Joseph~M. Landsberg - Giorgio Ottaviani, \emph{Equations for secant varieties
  of {V}eronese and other varieties}, Ann. Mat. Pura Appl. (4) {192} no.~4
  (2013), 569--606.

\bibitem{manivel2020complete}
Laurent Manivel - Mateusz Micha\l{}ek - Leonid Monin - Tim Seynnaeve - Martin
  Vodi\v{c}ka, \emph{Complete quadrics: Schubert calculus for Gaussian models
  and semidefinite programming}, arXiv:2011.08791 (2020).

\bibitem{ExponentialVarieties}
Mateusz Micha\l{}ek - Bernd Sturmfels - Caroline Uhler - Piotr Zwiernik,
  \emph{Exponential varieties}, Proc. Lond. Math. Soc. (3) {112} no.~1 (2016),
  27--56.

\bibitem{MPIrepo}
Tim Seynnaeve, \emph{Complete quadrics: Schubert calculus for Gaussian models
  and semidefinite programming},  (2020),
  (\url{https://mathrepo.mis.mpg.de/MLdegreeCompleteQuadrics}).

\bibitem{Sascha}
Bernd Sturmfels - Sascha Timme - Piotr Zwiernik, \emph{Estimating linear
  covariance models with numerical nonlinear algebra}, Algebraic Statistics
  {11} no.~1 (2020), 31–52.

\bibitem{Sturmfels2010Gaussian}
Bernd Sturmfels - Caroline Uhler, \emph{Multivariate {G}aussian, semidefinite
  matrix completion, and convex algebraic geometry}, Ann. Inst. Statist. Math.
  {62} no.~4 (2010), 603--638.

\bibitem{sylvester1970principles}
James~J. Sylvester, \emph{An Essay on Canonical Forms, Supplement to a Sketch
  of a Memoir on Elimination, Transformation and Canonical Forms},  (1851),
  Paper 34 in {\it Collected Mathematical Papers of James Joseph Sylvester},
  vol. I, 203--216, Chelsea, New York, 1973, originally published by Cambridge
  University Press, London, Fetter Lane, E.C. 1904.

\bibitem{terracini-1911-sulle-vk}
Alessandro Terracini, \emph{Sulle $v_k$ per cui la variet{\`a} degli $S_h\,(h +
  1)$ seganti ha dimensione minore dell'ordinario}, Rendiconti del Circolo
  matematico di Palermo (1884-1940) {31} no.~1 (1911), 392--396.

\bibitem{LinearCovarianceModels}
Sascha Timme, \emph{LinearCovarianceModels.jl},  (2020),
  (\url{https://github.com/saschatimme/LinearCovarianceModels.jl}).

\end{thebibliography}

\end{document}